\documentclass{amsart}

\newtheorem{theorem}{Theorem}[section]

\newtheorem{corollary}[theorem]{Corollary}

\theoremstyle{definition}
\newtheorem{definition}[theorem]{Definition}
\newtheorem{example}[theorem]{Example}

\theoremstyle{remark}
\newtheorem{remark}[theorem]{Remark}

\numberwithin{equation}{section}

\begin{document}

\title[Generalized fractional calculus with
applications]{Generalized fractional calculus with
applications to the calculus of variations}

\thanks{Submitted 22-Dec-2011;
revised 26-Jan-2012; accepted 27-Jan-2012;
for publication in \emph{Computers and Mathematics with Applications}.}

\thanks{Part of the first author's Ph.D., which is carried out
at the University of Aveiro under the Doctoral Programme
\emph{Mathematics and Applications}
of Universities of Aveiro and Minho.}


\author[T. Odzijewicz]{Tatiana Odzijewicz}

\address{Center for Research and Development in Mathematics and Applications\newline
\indent Department of Mathematics, University of Aveiro, 3810-193 Aveiro, Portugal}

\email{tatianao@ua.pt}


\author[A. B. Malinowska]{Agnieszka B. Malinowska}

\address{Faculty of Computer Science, Bia{\l}ystok University of Technology\newline
\indent 15-351 Bia\l ystok, Poland}

\email{a.malinowska@pb.edu.pl}


\author[D. F. M. Torres]{Delfim F. M. Torres}

\address{Center for Research and Development in Mathematics and Applications\newline
\indent Department of Mathematics, University of Aveiro, 3810-193 Aveiro, Portugal}

\email{delfim@ua.pt}



\date{}


\begin{abstract}
We study operators that are generalizations
of the classical Riemann--Liouville fractional integral,
and of the Riemann--Liouville and Caputo fractional derivatives.
A useful formula relating the generalized fractional derivatives
is proved, as well as three relations of fractional integration by parts
that change the parameter set of the given operator into its dual.
Such results are explored in the context of dynamic optimization,
by considering problems of the calculus of variations
with general fractional operators.
Necessary optimality conditions of Euler--Lagrange type
and natural boundary conditions for unconstrained and constrained
problems are investigated. Interesting results are obtained
even in the particular case when the generalized operators
are reduced to be the standard fractional derivatives
in the sense of Riemann--Liouville or Caputo.
As an application we provide a class of variational problems
with an arbitrary kernel that give answer
to the important coherence embedding problem.
Illustrative optimization problems are considered.
\end{abstract}




\maketitle


\section{Introduction}

Fractional calculus studies
derivatives (and integrals) of non-integer order.
It is a classical mathematical field
as old as calculus itself \cite{book:Kilbas}.
During almost 300 years, fractional calculus was considered
as pure mathematics, with nearly no applications. In recent years,
however, the situation changed dramatically,
with fractional calculus becoming an
interesting and useful topic among engineers
and applied scientists, and an excellent tool
for description of memory and heredity effects \cite{TM}.

One of the earliest applications
of fractional calculus was to construct
a complete mechanical description of nonconservative systems,
including Lagrangian and Hamiltonian mechanics
\cite{CD:Riewe:1996,CD:Riewe:1997}. Riewe's results
\cite{CD:Riewe:1996,CD:Riewe:1997}
mark the beginning of the fractional calculus of variations
and are of upmost importance: nonconservative and dissipative
processes are widespread in the physical world.
Fractional calculus provide the necessary tools
to apply variational principles to systems
characterized by friction or other dissipative
forces, being even possible to deduce fractional conservation laws
along the nonconservative extremals \cite{CD:FredericoTorres:2007}.

The theory of the calculus of variations with fractional
derivatives is nowadays under strong current development,
and the literature is already vast. We do not try
to make here a review.
Roughly speaking, available results in the literature
use different notions of fractional derivatives,
in the sense of Riemann--Liouville
\cite{LeitmannMethod,gastao,MyID:181},
Caputo \cite{Caputo,comBasia:Frac1,comDorota},
Riesz \cite{Agrawal2007,withGasta:SI:Leit:85,Rabei:et:al},
combined fractional derivatives \cite{Klimek02,Klimek05,Basia:Spain2010},
or modified/generalized versions of the classical fractional operators
\cite{MyID:182,Cresson,Jumarie,BasiaRachid,MyID:207},
in order to describe different variational principles.
Here we develop a more general perspective to the subject,
by considering three fractional operators that depend
on a general kernel. By choosing special cases for the kernel,
one obtains the standard fractional operators and previous results
in the literature. More important,
the general approach here considered
brings new insights and give answers to some important questions.

The text is organized as follows. In Section~\ref{sec:prelim}
the generalized fractional operators $K_P^\alpha$,
$A_P^\alpha$ and $B_P^\alpha$ are introduced
and basic results given. The main contributions of the paper
appear in Section~\ref{sec:mr}: we prove a useful relation
between $A_P^\alpha$ and $B_P^\alpha$ (Theorem~\ref{thm:rel}),
several formulas of integration by parts that change the parameter
set $P$ into its dual $P^*$ (Theorems~\ref{thm:gfip:Kop},
\ref{thm:IPL1} and \ref{thm:gfip}),
and new fractional necessary optimality conditions
for generalized variational problems with mixed
integer and fractional order derivatives and integrals
(Theorems~\ref{theorem:ELCaputo}, \ref{thm:tc:a} and \ref{theorem:EL2}).
We see that even for an optimization problem that does not depend on
generalized Riemann--Liouville fractional derivatives, such derivatives appear
naturally in the necessary optimality conditions. This
is connected with duality of operators
in the formulas of integration by parts
and explains no-coherence of the fractional embedding \cite{Cresson}.
This is addressed in Section~\ref{sec:coh}, where we give an answer
to the important question of coherence, by providing a class of
fractional variational problems that does not depend on the kernel,
for which the embedded Euler--Lagrange equation coincides with the one
obtained by the least action principle (Theorem~\ref{thm:coherent}).
Finally, some concrete examples of optimization problems
are discussed in Section~\ref{sec:ex}.


\section{Basic notions}
\label{sec:prelim}

Throughout the text, $\alpha$ denotes a positive
real number between zero and one,
and $\partial_{i} F$ the partial
derivative of a function $F$
with respect to its $i$th argument.
We consider the generalized fractional operators
$K_P^\alpha$, $A_P^\alpha$ and $B_P^\alpha$
as denoted in \cite{OmPrakashAgrawal}.
The study of generalized fractional operators
and their applications has a long and rich history.
We refer the reader to the book \cite{book:Kiryakova}.

\begin{definition}[Generalized fractional integral]
The operator $K_P^\alpha$ is given by
\begin{equation*}
K_P^{\alpha}f(t)=p\int\limits_{a}^{t}k_{\alpha}(t,\tau)f(\tau)d\tau
+q\int\limits_{t}^{b}k_{\alpha}(\tau,t)f(\tau)d\tau,
\end{equation*}
where $P=\langle a,t,b,p,q\rangle$ is the \emph{parameter set} ($p$-set for brevity),
$t\in[a,b]$, $p,q$ are real numbers, and $k_{\alpha}(t,\tau)$
is a kernel which may depend on $\alpha$.
The operator $K_P^\alpha$ is referred as the \emph{operator $K$}
($K$-op for simplicity) of order $\alpha$ and $p$-set $P$.
\end{definition}

\begin{theorem}
Let $\alpha\in(0,1)$ and $P=\langle a,t,b,p,q \rangle$. If $k_\alpha(\cdot,\cdot)$
is a square-integrable function on $\Delta=[a,b]\times[a,b]$, then
$K_P^{\alpha}:L_2\left([a,b]\right)\rightarrow L_2\left([a,b]\right)$
is a well defined bounded linear operator.
\end{theorem}

\begin{proof}
Let $\alpha\in(0,1)$ and $P=\langle a,t,b,p,q\rangle$. Define
\[
G(t,\tau):=
\left\{ \begin{array}{ll}
p k_\alpha(t,\tau) & \mbox{if $\tau < t$},\\
q k_\alpha(\tau,t) & \mbox{if $\tau \geq t$.}
\end{array} \right.
\]
For all $f\in L_2\left([a,b]\right)$ one has
$K_P^{\alpha} f(t)=\int_a^bG(t,\tau)f(\tau)d\tau$
with $G(t,\tau)\in L_2\left(\Delta\right)$. It is not difficult to see that
$K_P^\alpha$ is linear and $K_P^{\alpha} f\in L_2\left([a,b]\right)$ for all
$f\in L_2\left([a,b]\right)$. Moreover, applying the Cauchy--Schwarz inequality
and Fubini's theorem, we obtain
\begin{equation*}
\begin{split}
\left\|K_P^{\alpha} f\right\|_2^2
&=\int_a^b\left|\int_a^bG(t,\tau)f(\tau)d\tau\right|^2dt\\
&\leq\int_a^b\left[\left(\int_a^b\left|G(t,\tau)\right|^2
d\tau\right)\left(\int_a^b\left|f(\tau)\right|^2d\tau\right)\right]dt\\
&=\left\|f\right\|_2^2\int_a^b\int_a^b\left|G(t,\tau)\right|^2d\tau dt.
\end{split}
\end{equation*}
For $f\in L_2\left([a,b]\right)$ such that $\left\|f\right\|_2\leq 1$ we have
$\left\|K_P^{\alpha} f \right\|_2\leq\left(\int_a^b
\int_a^b\left|G(t,\tau)\right|^2d\tau dt\right)^{\frac{1}{2}}$.
Therefore, $\left\|K_P^{\alpha} \right\|_2\leq\left(\int_a^b\int_a^b
\left|G(t,\tau)\right|^2d\tau dt\right)^{\frac{1}{2}}$.
\end{proof}

\begin{theorem}
\label{theorem:L1}
Let $k_\alpha$ be a difference kernel, \textrm{i.e.},
$k_\alpha(t,\tau)=k_\alpha(t-\tau)$ and $k_\alpha\in L_1\left([a,b]\right)$.
Then $K_P^{\alpha}:L_1\left([a,b]\right)\rightarrow L_1\left([a,b]\right)$
is a well defined bounded linear operator.
\end{theorem}

\begin{proof}
Obviously, the operator is linear. Let $\alpha\in(0,1)$,
$P=\langle a,t,b,p,q\rangle$, and $f\in L_1\left([a,b]\right)$. Define
\[
F(\tau,t):=
\left\{
\begin{array}{ll}
p\left|k_\alpha(t-\tau)\right|\cdot \left|f(\tau)\right| & \mbox{if $\tau \leq t$}\\
q\left|k_\alpha(\tau-t)\right|\cdot \left|f(\tau)\right| & \mbox{if $\tau > t$}
\end{array} \right.
\]
for all $(\tau,t)\in\Delta=[a,b]\times[a,b]$.
Since $F$ is measurable on the square $\Delta$ we have
\begin{equation*}
\begin{split}
\int_a^b \left(\int_a^b F(\tau,t)dt\right)d\tau
&=\int_a^b\left[\left|f(\tau)\right|\left(\int_{\tau}^bp\left|k_\alpha(t-\tau)\right|dt
+\int_{a}^{\tau}q\left|k_\alpha(\tau-t)\right|dt\right)\right]d\tau\\
&\leq\int_a^b\left|f(\tau)\right|(p-q)\left\|k_\alpha\right\|d\tau\\
&=(p-q)\left\|k_\alpha\right\|\cdot\left\|f\right\|.
\end{split}
\end{equation*}
It follows from Fubini's theorem that $F$
is integrable on the square $\Delta$. Moreover,
\begin{equation*}
\begin{split}
\left\|K_P^\alpha f\right\|
&=\int_a^b\left|p\int_{a}^{t}k_{\alpha}(t-\tau)f(\tau)d\tau
+q\int_{t}^{b}k_{\alpha}(\tau-t)f(\tau)d\tau\right|dt\\
&\leq\int_a^b\left(p\int_{a}^{t}\left|
k_{\alpha}(t-\tau)\right|\cdot\left|f(\tau)\right|d\tau
+q\int_{t}^{b}\left|k_{\alpha}(\tau-t)\right|
\cdot\left|f(\tau)\right|d\tau\right)dt\\
&=\int_a^b\left(\int_a^b F(\tau,t)d\tau\right)dt\\
&\leq (p-q)\left\|k_\alpha\right\|\cdot\left\|f\right\|.
\end{split}
\end{equation*}
Hence, $K_P^{\alpha}:L_1\left([a,b]\right)\rightarrow
L_1\left([a,b]\right)$ and $\left\|K_P^\alpha
\right\|\leq(p-q)\left\|k_\alpha\right\|$.
\end{proof}

\begin{theorem}
\label{theorem:exist}
Let $k_{1-\alpha}$ be a difference kernel, \textrm{i.e.},
$k_{1-\alpha}(t,\tau)=k_{1-\alpha}(t-\tau)$ and $k_{1-\alpha}\in L_1\left([a,b]\right)$.
If $f\in AC\left([a,b]\right)$, then the $K$-op of order $1-\alpha$
and $p$-set $P=\langle a,t,b,p,q\rangle$, \textrm{i.e.},
\begin{equation*}
K_P^{1-\alpha}f(t)
=p\int\limits_{a}^{t}k_{1-\alpha}(t-\tau)f(\tau)d\tau
+q\int\limits_{t}^{b}k_{1-\alpha}(\tau-t)f(\tau)d\tau,
\end{equation*}
belongs to $AC\left([a,b]\right)$.
\end{theorem}

\begin{proof}
Let $P_1=\langle a,t,b,p,0\rangle$ and $P_2=\langle a,t,b,0,q\rangle$. Then,
$K_P^{1-\alpha}=K_{P_1}^{1-\alpha}+K_{P_2}^{1-\alpha}$.
First we show that $K_{P_1}^{1-\alpha}f \in AC\left([a,b]\right)$.
The condition $f\in AC\left([a,b]\right)$ implies
\begin{equation*}
f(x)=\int_a^x g(t)dt+f(a), \textnormal{ where } g\in L_1\left([a,b]\right).
\end{equation*}
Let $s=x-a$ and
$$
h(s)=\int_0^s k_{1-\alpha}(\tau)g(s+a-\tau)d\tau.
$$
Integrating,
$$
\int_0^s h(\theta)d\theta
=\int_0^s d\theta \int_0^{\theta}k_{1-\alpha}(\tau)g(\theta+a-\tau)d\tau,
$$
and changing the order of integration we obtain
$$
\int_0^s h(\theta)d\theta =\int_0^s d\tau \int_\tau^s k_{1-\alpha}(\tau)g(\theta+a-\tau)d\theta
=\int_0^s k_{1-\alpha}(\tau) d\tau \int_\tau^s g(\theta+a-\tau)d\theta.
$$
Putting $\xi=\theta+a-\tau$ and $d\xi=d\theta$, we have
$$
\int_0^s h(\theta)d\theta=\int_0^s k_{1-\alpha}(\tau) d\tau \int_a^{x-\tau} g(\xi)d\xi.
$$
Because $\displaystyle \int_a^{x-\tau} g(\xi)d\xi=f(x-\tau)-f(a)$, the following equality holds:
$$
\int_0^s h(\theta)d\theta=\int_0^s k_{1-\alpha}(\tau)f(x-\tau) d\tau
-f(a)\int_0^s k_{1-\alpha}(\tau) d\tau,
$$
that is,
$$
\int_0^s k_{1-\alpha}(\tau)f(x-\tau) d\tau
=\int_0^s h(\theta)d\theta+f(a)\int_0^s k_{1-\alpha}(\tau) d\tau.
$$
Both functions on the right-hand side of the equality belong to $AC\left([a,b]\right)$.
Hence,
$$
\int_0^s k_{1-\alpha}(\tau)f(x-\tau) d\tau \in AC\left([a,b]\right).
$$
Substituting $t=x-\tau$ and $dt=-d\tau$, we get
$$
\int_a^x k_{1-\alpha}(x-t)f(t) dt \in AC\left([a,b]\right).
$$
This means that $K_{P_1}^{1-\alpha}f \in AC\left([a,b]\right)$.
The proof that $K_{P_2}^{1-\alpha}f \in AC\left([a,b]\right)$
is analogous, and since the sum of two absolutely continuous functions
is absolutely continuous, it follows that
$K_{P}^{1-\alpha}f \in AC\left([a,b]\right)$.
\end{proof}

\begin{remark}
The $K$-op reduces to the classical left or right Riemann--Liouville
fractional integral (see, \textrm{e.g.}, \cite{book:Kilbas,book:Podlubny})
for a suitably chosen kernel
$k_{\alpha}(t,\tau)$ and $p$-set $P$. Indeed,
let $k_{\alpha}(t-\tau)=\frac{1}{\Gamma(\alpha)}(t-\tau)^{\alpha-1}$.
If $P=\langle a,t,b,1,0\rangle$, then
\begin{equation}
\label{eq:class:left:int}
K_{P}^{\alpha}f(t)=\frac{1}{\Gamma(\alpha)}
\int\limits_a^t(t-\tau)^{\alpha-1}f(\tau)d\tau
=: {_{a}}\textsl{I}^{\alpha}_{t} f(t)
\end{equation}
is the left Riemann--Liouville fractional integral
of order $\alpha$; if $P=\langle a,t,b,0,1\rangle$, then
\begin{equation}
\label{eq:class:right:int}
K_{P}^{\alpha}f(t)=\frac{1}{\Gamma(\alpha)}
\int\limits_t^b(\tau-t)^{\alpha-1}f(\tau)d\tau
=: {_{t}}\textsl{I}^{\alpha}_{b} f(t)
\end{equation}
is the right Riemann--Liouville fractional integral
of order $\alpha$. Theorem~\ref{theorem:L1}
with $k_{\alpha}(t-\tau)=\frac{1}{\Gamma(\alpha)}(t-\tau)^{\alpha-1}$
asserts the well-known fact that the
Riemann--Liouville fractional integrals
${_{a}}\textsl{I}^{\alpha}_{t}$,
${_{t}}\textsl{I}^{\alpha}_{b}:
L_1\left([a,b]\right)\rightarrow
L_1\left([a,b]\right)$
given by \eqref{eq:class:left:int}
and \eqref{eq:class:right:int}
are well defined bounded linear operators.
\end{remark}

The fractional derivatives $A_P^\alpha$
and $B_P^\alpha$ are defined with the help of
the generalized fractional integral $K$-op.

\begin{definition}[Generalized Riemann--Liouville fractional derivative]
\label{def:GRL}
Let $P$ be a given parameter set. The operator $A_P^\alpha$,
$0 < \alpha < 1$, is defined by $A_P^\alpha = D \circ K_P^{1-\alpha}$,
where $D$ denotes the standard derivative.
We refer to $A_P^\alpha$ as \emph{operator $A$} ($A$-op)
of order $\alpha$ and $p$-set $P$.
\end{definition}

A different fractional derivative is obtained by interchanging the order
of the operators in the composition that defines $A_P^\alpha$.

\begin{definition}[Generalized Caputo fractional derivative]
\label{def:GC}
Let $P$ be a given parameter set. The operator $B_P^\alpha$,
$\alpha \in (0,1)$, is defined by $B_P^\alpha =K_P^{1-\alpha} \circ D$
and is referred as the \emph{operator $B$} ($B$-op)
of order $\alpha$ and $p$-set $P$.
\end{definition}

\begin{remark}
The operator $B_P^\alpha$ is defined
for absolute continuous functions
$f\in AC\left([a,b]\right)$,
while the operator $A_P^\alpha$ acts on the bigger class
of functions $f$ such that
$K_P^{1-\alpha} f\in AC\left([a,b]\right)$.
\end{remark}

\begin{remark}
The standard Riemann--Liouville and Caputo fractional derivatives
(see, \textrm{e.g.}, \cite{book:Kilbas,book:Podlubny}) are easily obtained
from the generalized operators $A_P^\alpha $ and $B_P^\alpha$, respectively.
Let $k_{1-\alpha}(t-\tau)=\frac{1}{\Gamma(1-\alpha)}(t-\tau)^{-\alpha}$,
$\alpha \in (0,1)$. If $P=\langle a,t,b,1,0\rangle$, then
\begin{equation*}
A_{P}^\alpha f(t)=\frac{1}{\Gamma(1-\alpha)}
\frac{d}{dt} \int\limits_a^t(t-\tau)^{-\alpha}f(\tau)d\tau
=: {_{a}}\textsl{D}^{\alpha}_{t} f(t)
\end{equation*}
is the standard left Riemann--Liouville fractional derivative
of order $\alpha$ while
\begin{equation*}
B_{P}^\alpha f(t)=\frac{1}{\Gamma(1-\alpha)}
\int\limits_a^t(t-\tau)^{-\alpha} f'(\tau)d\tau
=: {^{C}_{a}}\textsl{D}^{\alpha}_{t} f(t)
\end{equation*}
is the standard left Caputo fractional derivative of order $\alpha$;
if $P=\langle a,t,b,0,1\rangle$, then
\begin{equation*}
- A_{P}^\alpha f(t)
=- \frac{1}{\Gamma(1-\alpha)} \frac{d}{dt}
\int\limits_t^b(\tau-t)^{-\alpha}f(\tau)d\tau
=: {_{t}}\textsl{D}^{\alpha}_{b} f(t)
\end{equation*}
is the standard right Riemann--Liouville
fractional derivative of order $\alpha$ while
\begin{equation*}
- B_{P}^\alpha f(t) = - \frac{1}{\Gamma(1-\alpha)}
\int\limits_t^b(\tau-t)^{-\alpha} f'(\tau)d\tau
=: {^{C}_{t}}\textsl{D}^{\alpha}_{b} f(t)
\end{equation*}
is the standard right Caputo fractional derivative of order $\alpha$.
\end{remark}


\section{Main results}
\label{sec:mr}

We begin by proving in Section~\ref{sec:rel}
that for a certain class of kernels there exists
a direct relation between the fractional derivatives
$A_P^\alpha$ and $B_P^\alpha$ (Theorem~\ref{thm:rel}).
Section~\ref{sec:GFIP} gives integration by parts formulas
for the generalized fractional setting
(Theorems~\ref{thm:gfip:Kop}, \ref{thm:IPL1} and \ref{thm:gfip}).
Section~\ref{sub:sec:new:EL:Cp} is devoted to
variational problems with generalized fractional-order operators.
New results include necessary optimality conditions of
Euler--Lagrange type for unconstrained (Theorem~\ref{theorem:ELCaputo})
and constrained problems (Theorem~\ref{theorem:EL2}),
and a general transversality condition (Theorem~\ref{thm:tc:a}).
Interesting results are obtained as particular cases.
Finally, in Section~\ref{sec:coh} we provide a class
of generalized fractional problems of the calculus of variations
for which one has a coherent embedding,
compatible with the least action principle (Theorem~\ref{thm:coherent}).
This provides a general answer to an open question posed in \cite{Cresson}.


\subsection{A relation between operators $A$ and $B$}
\label{sec:rel}

Next theorem gives a useful relation
between $A$-op and $B$-op. In the calculus of variations,
equality \eqref{eq:relAB}
can be used to provide a necessary optimality
condition involving the same operators
as in the data of the optimization problem
(\textrm{cf.} Remark~\ref{rem:AintoB}
of Section~\ref{sub:sec:new:EL:Cp}).

\begin{theorem}
\label{thm:rel}
Let $0<\alpha<1$, $P=\langle a,t,b,p,q\rangle$,
and $y\in AC\left([a,b]\right)$.
If kernel $k_{1-\alpha}$ is integrable
and there exist functions $f$ and $g$ such that
\begin{equation}
\label{eq:m:h}
\int_a^t k_{1-\alpha}(\theta,\tau) d\theta
+ \int_a^\tau k_{1-\alpha}(t,\theta) d\theta = g(t)+f(\tau)
\end{equation}
for all $t,\tau \in[a,b]$,
then the following relation holds:
\begin{equation}
\label{eq:relAB}
A_P^\alpha y(t)
=p y(a)k_{1-\alpha}(t,a)-q y(b)k_{1-\alpha}(b,t)
+B_P^\alpha y(t)
\end{equation}
for all $t\in[a,b]$.
\end{theorem}

\begin{proof}
Let $h_{1-\alpha}$ be defined by
$h_{1-\alpha}(t,\tau) := \int_a^\tau k_{1-\alpha}(t,\theta) d\theta -g(t)$.
Then, by hypothesis \eqref{eq:m:h},
$\partial_2 h_{1-\alpha} = -\partial_1 h_{1-\alpha}=k_{1-\alpha}$.
We obtain the intended conclusion from the definition of $A$-op and $B$-op,
integrating by parts, and differentiating:
\begin{equation*}
\begin{split}
A_P^\alpha &y(t)=\frac{d}{d t} K_P^{1-\alpha} y(t)
=\frac{d}{d t}\left\{ p\int_a^t k_{1-\alpha}(t,\tau)y(\tau)d\tau
+q\int_t^b k_{1-\alpha}(\tau,t)y(\tau)d\tau \right\}\\
&=\frac{d}{d t}\Biggl\{ \left.p y(t)
h_{1-\alpha}(t,\tau)\right|_{\tau=t}-p y(a)h_{1-\alpha}(t,a)
-p \int_a^t h_{1-\alpha}(t,\tau)\frac{d}{d\tau}y(\tau)d\tau\\
&\quad -q y(b)h_{1-\alpha}(b,t)+\left.q y(t)h_{1-\alpha}(\tau,t)\right|_{\tau=t}
+q \int_t^b h_{1-\alpha}(\tau,t)\frac{d}{d\tau}y(\tau)d\tau\Biggr\}\\
&=\left.p y(t)\frac{d}{d t}h_{1-\alpha}(t,t+\epsilon)\right|_{\epsilon=0}
-p y(a)\partial_1 h_{1-\alpha}(t,a)
- p \int_a^t \partial_1 h_{1-\alpha}(t,\tau)\frac{d}{d\tau}y(\tau)d\tau\\
&\quad -q y(b)\partial_2 h_{1-\alpha}(b,t)
+\left.q y(t)\frac{d}{dt}h_{1-\alpha}(t+\epsilon,t)\right|_{\epsilon=0}
+q \int_t^b \partial_2 h_{1-\alpha}(\tau,t)\frac{d}{d\tau}y(\tau)d\tau\\
&=p y(a)k_{1-\alpha}(t,a)-q y(b)k_{1-\alpha}(b,t)+B_P^\alpha y(t).
\end{split}
\end{equation*}
\end{proof}

\begin{example}
Let $k_{1-\alpha}(t-\tau)=\frac{1}{\Gamma(1-\alpha)}(t-\tau)^{-\alpha}$.
Simple calculations show that \eqref{eq:m:h} is satisfied.
If $P=\langle a,t,b,1,0\rangle$, then \eqref{eq:relAB}
reduces to the relation
\begin{equation*}
^{C}_{a}D^{\alpha}_{t}y(t)
= {_{a}D^{\alpha}_{t}}y(t)
-\frac{y(a)}{\Gamma(1-\alpha)}(t-a)^{-\alpha}
\end{equation*}
between the left Riemann--Liouville fractional derivative ${_{a}D^{\alpha}_{t}}$
and the left Caputo fractional derivative $^{C}_{a}D^{\alpha}_{t}$;
if $P=\langle a,t,b,0,1\rangle$, then we get the relation
\begin{equation*}
^{C}_{t}D^{\alpha}_{b}y(t)
= {_{t}D^{\alpha}_{b}}y(t)
-\frac{y(b)}{\Gamma(1-\alpha)}(b-t)^{-\alpha}
\end{equation*}
between the right Riemann--Liouville fractional derivative ${_{t}D^{\alpha}_{b}}$
and the right Caputo fractional derivative $^{C}_{t}D^{\alpha}_{b}$.
\end{example}


\subsection{Fractional integration by parts}
\label{sec:GFIP}

The proof of Theorem~\ref{thm:rel} uses one basic
but important technique of classical integral calculus:
integration by parts. In this section we obtain several
formulas of integration by parts
for the generalized fractional calculus.
Our results are particularly useful
with respect to applications in dynamic optimization
(\textrm{cf.} Section~\ref{sub:sec:new:EL:Cp}),
where the derivation of the Euler--Lagrange equations
uses, as a key step in the proof, integration by parts.

In our setting, integration by parts
changes a given $p$-set $P$ into its dual $P^{*}$.
The term \emph{duality} comes from the fact that
$P^{**} =P$.

\begin{definition}[Dual $p$-set]
Given a $p$-set $P=\langle a,t,b,p,q\rangle$
we denote by $P^{*}$ the $p$-set
$P^{*} = \langle a,t,b,q,p\rangle$.
We say that $P^{*}$ is the dual of $P$.
\end{definition}

Our first formula of fractional integration by parts
involves the $K$-op.

\begin{theorem}[Fractional integration by parts for the $K$-op]
\label{thm:gfip:Kop}
Let $\alpha \in (0,1)$, $P=\langle a,t,b,p,q\rangle$,
$k_{\alpha}$ be a square-integrable function
on $\Delta=[a,b]\times[a,b]$, and $f,g\in L_2\left([a,b]\right)$.
The generalized fractional integral satisfies
the integration by parts formula
\begin{equation}
\label{eq:fracIP:K}
\int\limits_a^b g(t)K_P^{\alpha}f(t)dt
=\int\limits_a^b f(t)K_{P^*}^{\alpha}g(t)dt,
\end{equation}
where $P^*$ is the dual of $P$.
\end{theorem}

\begin{proof}
Let $\alpha\in(0,1)$, $P=\langle a,t,b,p,q\rangle$,
and $f,g\in L_2\left([a,b]\right)$. Define
\[
F(\tau,t):=
\left\{
\begin{array}{ll}
\left|pk_\alpha(t,\tau)\right|
\cdot\left|g(t)\right|\cdot \left|f(\tau)\right|
& \mbox{if $\tau \leq t$}\\
\left|qk_\alpha(\tau,t)\right|
\cdot \left|g(t)\right|\cdot\left|f(\tau)\right|
& \mbox{if $\tau > t$}
\end{array}\right.
\]
for all $(\tau,t)\in\Delta$. Then, applying Holder's inequality, we obtain
\begin{equation*}
\begin{split}
\int_a^b &\left(\int_a^b F(\tau,t)dt\right)d\tau\\
&=\int_a^b\left[\left|f(\tau)\right|\left(\int_{\tau}^b\left|pk_\alpha(t,\tau)\right|
\cdot\left|g(t)\right|dt+\int_{a}^{\tau}\left|qk_\alpha(\tau,t)\right|
\cdot\left|g(t)\right|dt\right)\right]d\tau\\
&\leq \int_a^b\left[\left|f(\tau)\right|\left(
\int_{a}^b\left|pk_\alpha(t,\tau)\right|\cdot
\left|g(t)\right|dt+\int_{a}^{b}\left|qk_\alpha(\tau,t)\right|
\cdot\left|g(t)\right|dt\right)\right]d\tau\\
&\leq \int_a^b\left\{\left|f(\tau)\right|\left[\left(
\int_a^b\left|pk_\alpha(t,\tau)\right|^2dt\right)^{\frac{1}{2}}\left(
\int_a^b\left|g(t)\right|^2dt\right)^{\frac{1}{2}}\right.\right.\\
&\qquad\qquad \left.\left.+\left(\int_a^b\left|qk_\alpha(\tau,t)\right|^2
dt\right)^{\frac{1}{2}}\left(\int_a^b\left|g(t)\right|^2
dt\right)^{\frac{1}{2}}\right]\right\}d\tau.
\end{split}
\end{equation*}
By Fubini's theorem, functions $k_{\alpha,\tau}(t):=k_{\alpha}(t,\tau)$
and $\hat{k}_{\alpha,\tau}(t):=k_{\alpha}(\tau,t)$
belong to $L_2\left([a,b]\right)$ for almost all $\tau\in[a,b]$. Therefore,
\begin{equation*}
\begin{split}
\int_a^b &\left\{\left|f(\tau)\right|\left[\left(\int_a^b\left|pk_\alpha(t,\tau)\right|^2dt
\right)^{\frac{1}{2}}\left(\int_a^b\left|g(t)\right|^2dt\right)^{\frac{1}{2}}\right.\right.\\
&\qquad\left.\left.+\left(\int_a^b\left|qk_\alpha(\tau,t)\right|^2dt\right)^{\frac{1}{2}}\left(
\int_a^b\left|g(t)\right|^2dt\right)^{\frac{1}{2}}\right]\right\}d\tau\\
&=\left\|g\right\|_2\int_a^b\left[\left|f(\tau)\right|\left(\left\|pk_{\alpha,\tau}\right\|_2
+\left\|q\hat{k}_{\alpha,\tau}\right\|_2\right)\right]d\tau\\
&\leq \left\|g\right\|_2 \left(\int_a^b\left|f(\tau)\right|^2d\tau\right)^{\frac{1}{2}}\left(
\int_a^b\left|\left\|pk_{\alpha,\tau}\right\|_2+\left\|q\hat{k}_{\alpha,\tau}\right\|_2\right|^2
d\tau\right)^{\frac{1}{2}}\\
&\leq\left\|g\right\|_2\cdot\left\|f\right\|_2\left(\left\|pk_\alpha\right\|_2
+\left\|q k_\alpha\right\|_2\right) < \infty.
\end{split}
\end{equation*}
Hence, we can use again Fubini's theorem
to change the order of integration:
\begin{equation*}
\begin{split}
\int\limits_a^b g(t)K_P^{\alpha}f(t)dt
&= p\int\limits_a^b g(t)dt\int\limits_a^t f(\tau)k_{\alpha}(t,\tau)d\tau
+q\int\limits_a^b g(t)dt\int\limits_t^b f(\tau)k_{\alpha}(\tau,t)d\tau\\
&= p\int\limits_a^b f(\tau)d\tau\int\limits_{\tau}^b g(t)k_{\alpha}(t,\tau)dt
+q\int\limits_a^b f(\tau)d\tau\int\limits_a^{\tau} g(t)k_{\alpha}(\tau,t)dt\\
&=\int\limits_a^b f(\tau)K_{P^*}^{\alpha}g(\tau)d\tau.
\end{split}
\end{equation*}
\end{proof}

Next example shows that one cannot relax
the hypotheses of Theorem~\ref{thm:gfip:Kop}.

\begin{example}
Let $P=\langle 0,t,1,1,-1\rangle$, $f(t)=g(t)\equiv 1$,
and $k_{\alpha}(t,\tau)=\frac{t^2-\tau^2}{(t^2+\tau^2)^2}$.
Direct calculations show that
\begin{equation*}
\begin{split}
\int_0^1 K_P^{\alpha} 1 dt
&=\int_0^1\left(\int_0^t\frac{t^2-\tau^2}{(t^2+\tau^2)^2}d\tau
-\int_t^1\frac{\tau^2-t^2}{(t^2+\tau^2)^2}d\tau\right)dt\\
&=\int_0^1\left(\int_0^1\frac{t^2-\tau^2}{(t^2+\tau^2)^2}d\tau\right)dt
=\int_0^1 \frac{1}{t^2+1}dt=\frac{\pi}{4}
\end{split}
\end{equation*}
and
\begin{equation*}
\begin{split}
\int_0^1 K_{P^*}^{\alpha} 1 d\tau
&=\int_0^1\left(-\int_0^\tau\frac{\tau^2-t^2}{(t^2+\tau^2)^2}dt
+\int_\tau^1\frac{t^2-\tau^2}{(t^2+\tau^2)^2}dt\right)d\tau\\
&=-\int_0^1\left(\int_0^1\frac{\tau^2-t^2}{(t^2+\tau^2)^2}dt\right)d\tau
=-\int_0^1 \frac{1}{\tau^2+1}d\tau=-\frac{\pi}{4}.
\end{split}
\end{equation*}
Therefore, the integration by parts formula \eqref{eq:fracIP:K} does not hold.
Observe that in this case
$\int_0^1 \int_0^1 \left|k_\alpha(t,\tau)\right|^2 d\tau dt=\infty$.
\end{example}

For the classical Riemann--Liouville fractional integrals
the following result holds.

\begin{corollary}
Let $\frac{1}{2}<\alpha<1$. If $f, g\in L_2([a,b])$, then
\begin{equation}
\label{eq:PartsFrac}
\int_{a}^{b} g(t){_aI_t^\alpha}f(t)dt =\int_a^b f(t){_t
I_b^\alpha} g(t)dt.
\end{equation}
\end{corollary}

\begin{proof}
Let $k_{\alpha}(t,\tau)=\frac{1}{\Gamma(\alpha)}(t-\tau)^{\alpha-1}$.
For $\alpha\in\left(\frac{1}{2},1\right)$,
$k_\alpha$ is a square-integrable function
on $\Delta$ (see, \textrm{e.g.}, \cite[Theorem~4]{HardyLittlewood}).
Therefore, \eqref{eq:PartsFrac} follows from \eqref{eq:fracIP:K}.
\end{proof}

\begin{theorem}
\label{thm:IPL1}
Let $0<\alpha<1$ and $P=\langle a,t,b,p,q\rangle$.
If $k_{\alpha}(t,\tau)=k_{\alpha}(t-\tau)$,
$k_{\alpha}\in L_1\left([a,b]\right)$,
and $f,g\in C\left([a,b]\right)$, then
the integration by parts formula
\eqref{eq:fracIP:K} holds.
\end{theorem}

\begin{proof}
Let $\alpha\in(0,1)$, $P=\langle a,t,b,p,q\rangle$,
and $f,g\in C\left([a,b]\right)$. Define
\[
F(\tau,t):= \left\{
\begin{array}{ll}
\left|pk_\alpha(t-\tau)\right|\cdot\left|g(t)\right|
\cdot \left|f(\tau)\right| & \mbox{if $\tau \leq t$}\\
\left|qk_\alpha(\tau-t)\right|\cdot \left|g(t)\right|
\cdot\left|f(\tau)\right| & \mbox{if $\tau > t$}
\end{array}\right.
\]
for all $(\tau,t)\in\Delta$. Since $f$ and $g$
are continuous functions on $[a,b]$, they are bounded on $[a,b]$,
\textrm{i.e.}, there exist real numbers $C_1,C_2>0$ such that
$\left|g(t)\right| \leq C_1$ and
$\left|f(t)\right|\leq C_2$
for all $t\in [a,b]$. Therefore,
\begin{equation*}
\begin{split}
\int_a^b&\left(\int_a^b F(\tau,t)dt\right)d\tau\\
&=\int_a^b\left[\left|f(\tau)\right|\left(
\int_{\tau}^b\left|pk_\alpha(t-\tau)\right|\cdot\left|g(t)\right|dt
+\int_{a}^{\tau}\left|qk_\alpha(\tau-t)\right|
\cdot\left|g(t)\right|dt\right)\right]d\tau\\
&\leq \int_a^b\left[\left|f(\tau)\right|\left(
\int_{a}^b\left|pk_\alpha(t-\tau)\right|
\cdot\left|g(t)\right|dt+\int_{a}^{b}\left|qk_\alpha(\tau-t)\right|
\cdot\left|g(t)\right|dt\right)\right]d\tau\\
&\leq C_1 C_2\int_a^b\left(\int_{a}^b\left|pk_\alpha(t-\tau)\right|dt
+\int_{a}^{b}\left|qk_\alpha(\tau-t)\right|dt\right)d\tau\\
&=C_1 C_2 (\left|p\right|-\left|q\right|)\left\|k_\alpha\right\|(b-a)<\infty.
\end{split}
\end{equation*}
Hence, we can use Fubini's theorem to change
the order of integration in the iterated integrals.
\end{proof}

The next theorem follows from the classical formula
of integration by parts and fractional integration
by parts for the $K$-op.

\begin{theorem}[Fractional integration by parts for $A$-op and $B$-op]
\label{thm:gfip}
Let $\alpha \in (0,1)$ and $P=\langle a,t,b,p,q\rangle$.
If $f,g \in AC([a,b])$, then
\begin{gather}
\int\limits_a^b  g(t)A_{P}^\alpha f(t)dt
=\left. g(t) K_{P}^{1-\alpha}f(t)\right|_a^b
-\int_a^b f(t)B_{P^*}^\alpha g(t)dt,\label{eq:fip:1}\\
\int\limits_a^b g(t)B_{P}^\alpha f(t)dt
=\left. f(t) K_{P^*}^{1-\alpha}g(t)\right|_a^b
-\int_a^b f(t) A_{P^*}^\alpha g(t)dt.\label{eq:fip:2}
\end{gather}
\end{theorem}

\begin{proof}
From Definition~\ref{def:GRL} one has
$A_P^\alpha f(t) = D K_P^{1-\alpha} f(t)$.
Therefore,
\begin{equation*}
\begin{split}
\int_a^b g(t) A_{P}^\alpha f(t) dt
&= \int_a^b g(t) D K_P^{1-\alpha} f(t) dt\\
&= \left. g(t) K_P^{1-\alpha} f(t) \right|_a^b
- \int_a^b D g(t) K_P^{1-\alpha} f(t) dt,
\end{split}
\end{equation*}
where the second equality follows by the
standard integration by parts formula.
From \eqref{eq:fracIP:K} of Theorem~\ref{thm:gfip:Kop}
it follows the desired equality \eqref{eq:fip:1}:
\begin{equation*}
\begin{split}
\int_a^b g(t) A_{P}^\alpha f(t) dt
&= \left. g(t) K_{P}^{1-\alpha}f(t)\right|_a^b
-\int_a^b f(t) K_{P^*}^{1-\alpha} D g(t)dt.
\end{split}
\end{equation*}
We now prove \eqref{eq:fip:2}.
From Definition~\ref{def:GC} we know that
$B_P^\alpha f(t) = K_P^{1-\alpha} D f(t)$.
It follows that
$$
\int_a^b g(t) B_{P}^\alpha f(t)dt
= \int_a^b g(t) K_{P}^{1-\alpha} D f(t) dt.
$$
By Theorem~\ref{thm:gfip:Kop}
$$
\int_a^b g(t) B_{P}^\alpha f(t)dt
= \int_a^b D f(t) K_{P^*}^{1-\alpha} g(t) dt.
$$
The standard integration by parts formula
implies relation \eqref{eq:fip:2}:
$$
\int_a^b g(t) B_{P}^\alpha f(t)dt
=\left. f(t) K_{P^*}^{1-\alpha} g(t) \right|_a^b
-\int_a^b f(t) D K_{P^*}^{1-\alpha} g(t)dt.
$$
\end{proof}

\begin{corollary}
Let $0<\alpha<1$. If $f, g \in AC([a,b])$, then
\begin{equation*}
\begin{split}
\int_{a}^{b}  g(t) \, {^C_aD_t^\alpha}f(t)dt &=\left.f(t){_t
I_b^{1-\alpha}} g(t)\right|^{t=b}_{t=a}+\int_a^b f(t){_t D_b^\alpha}
g(t)dt,\\
\int_a^b f(t){_a D_t^\alpha} g(t)dt
&= \left. f(t){_a I_t^{1-\alpha}} g(t)\right|^{t=b}_{t=a}
+ \int_{a}^{b}  g(t) \, {^C_t D_b^\alpha}f(t)dt .
\end{split}
\end{equation*}
\end{corollary}


\subsection{Fractional variational problems}
\label{sub:sec:new:EL:Cp}

We study variational functionals with a Lagrangian depending
on generalized Caputo fractional derivatives
as well as derivatives of integer order.
Note that the only possibility of obtaining $y'$ from
$B_P^\alpha y$ or $A_P^\alpha y$ is to take the limit when $\alpha$ tends to one
but, in general, such a limit does not exist \cite{Ross:Samko:Love}.
Moreover, our Lagrangians may also depend on generalized fractional integrals.
This last possibility is used in Section~\ref{sec:coh} to solve
the important \emph{coherence problem}.

Our proofs are easily adapted to the cases when one considers
Riemann--Liouville $A$-op derivatives instead of
Caputo $B$-op derivatives, and vector admissible
functions $y$ instead of scalar ones.
Such versions are left to the reader.


\subsubsection{Fundamental problem}
\label{sec:fp}

We consider the problem of extremizing
(minimizing or maximizing) the functional
\begin{equation}
\label{eq:31}
\mathcal{J}[y]=\int\limits_a^b
F\left(t,y(t),y'(t),B_{P_1}^\alpha y(t),K_{P_2}^\beta y(t)\right)dt
\end{equation}
subject to boundary conditions
\begin{equation}
\label{eq:32}
y(a)=y_a, \quad y(b)=y_b,
\end{equation}
where $\alpha,\beta\in(0,1)$
and $P_j=\langle a,t,b,p_j,q_j\rangle$, $j=1,2$.

\begin{definition}
A Lipschitz function $y\in Lip\left([a,b];\mathbb{R}\right)$ is said to be
admissible for the fractional variational problem \eqref{eq:31}--\eqref{eq:32}
if it satisfies the given boundary conditions \eqref{eq:32}.
\end{definition}

For simplicity of notation we introduce the operator
$\left\{ \cdot \right\}_{P_1, P_2}^{\alpha,\beta}$ defined by
\begin{equation*}
\left\{y\right\}_{P_1, P_2}^{\alpha,\beta}(t)
=\left(t,y(t),y'(t),B_{P_1}^\alpha y(t),K_{P_2}^\beta y(t)\right).
\end{equation*}
We can then write \eqref{eq:31} in the form
$\mathcal{J}[y]=\int\limits_a^b
F\left\{y\right\}_{P_1, P_2}^{\alpha,\beta}(t) dt$.
We assume that $F\in C^1\left([a,b]\times\mathbb{R}^4;\mathbb{R}\right)$,
$t \mapsto \partial_4 F\left\{y\right\}_{P_1,P_2}^{\alpha,\beta}(t)$
is absolutely continuous and has a continuous derivative $A_{P_1^*}^\alpha$,
and $t \mapsto \partial_3 F\left\{y\right\}_{P_1,P_2}^{\alpha,\beta}(t)$
has a continuous derivative $\frac{d}{dt}$.

Next result gives a necessary optimality condition of Euler--Lagrange
type for problem \eqref{eq:31}--\eqref{eq:32}.

\begin{theorem}
\label{theorem:ELCaputo}
Let $y$ be a solution to problem \eqref{eq:31}--\eqref{eq:32}.
Then, $y$ satisfies the generalized Euler--Lagrange equation
\begin{multline}
\label{eq:eqELCaputo}
\partial_2 F \left\{y\right\}_{P_1, P_2}^{\alpha,\beta}(t)
-\frac{d}{dt}\partial_3 F\left\{y\right\}_{P_1, P_2}^{\alpha,\beta}(t)
-A_{P_1^*}^\alpha\partial_4 F\left\{y\right\}_{P_1, P_2}^{\alpha,\beta}(t)\\
+K_{P_2^*}^\beta\partial_5 F\left\{y\right\}_{P_1, P_2}^{\alpha,\beta}(t)=0
\end{multline}
for $t\in[a,b]$.
\end{theorem}

\begin{proof}
Suppose that $y$ is an extremizer of $\mathcal{J}$.
Consider the value of $\mathcal{J}$ at a nearby admissible function
$\hat{y}(t)=y(t)+\varepsilon\eta(t)$,
where $\varepsilon\in\mathbb{R}$ is a small parameter
and $\eta\in Lip\left([a,b];\mathbb{R}\right)$ is an arbitrary
function satisfying $\eta(a)=\eta(b)=0$. Let
$J(\varepsilon) := \mathcal{J}[\hat{y}]
=\int\limits_a^b F\left\{y+\varepsilon\eta\right\}_{P_1, P_2}^{\alpha,\beta}(t) dt$.
A necessary condition for $y$ to be an extremizer is given by $J'(0) = 0$, \textrm{i.e.},
\begin{multline}
\label{eq:33}
\int\limits_a^b
\Biggl(\partial_2 F \left\{y\right\}_{P_1, P_2}^{\alpha,\beta}(t)\cdot\eta(t)
+\partial_3 F \left\{y\right\}_{P_1, P_2}^{\alpha,\beta}(t)\frac{d}{dt}\eta(t)\\
+ \partial_4 F\left\{y\right\}_{P_1, P_2}^{\alpha,\beta}(t) B_{P_1}^\alpha\eta(t)
+\partial_5 F \left\{y\right\}_{P_1, P_2}^{\alpha,\beta}(t)\cdot K_{P_2}^\beta \eta(t)\Biggr)dt = 0.
\end{multline}
Using the classical integration by parts formula as well as our generalized fractional
versions (Theorems \ref{thm:gfip:Kop}, \ref{thm:IPL1} and \ref{thm:gfip}) we obtain that
\begin{equation*}
\int_a^b\partial_3F\frac{d\eta}{dt}dt=\left.\partial_3F\eta\right|_a^b
-\int_a^b\left(\eta\frac{d}{dt}\partial_3F\right)dt,
\end{equation*}
\begin{equation*}
\int\limits_a^b\partial_4 F B_{P_1}^\alpha\eta dt
=-\int\limits_a^b\eta A_{P_1^*}^\alpha\partial_4 F dt
+\left.\eta K_{P_1^*}^{1-\alpha}\partial_4 F\right|_a^b ,
\end{equation*}
and
\begin{equation*}
\int\limits_a^b\partial_5 F K_{P_2}^\beta\eta dt
=\int\limits_a^b
\eta K_{P_2^*}^\beta\partial_5 F dt,
\end{equation*}
where $P_j^*=\langle a,t,b,q_j,p_j\rangle$, $j=1,2$,
is the dual of $P_j$. Because $\eta(a)=\eta(b)=0$,
\eqref{eq:33} simplifies to
\begin{multline*}
\int_a^b\eta(t)\Biggl(\partial_2 F \left\{y\right\}_{P_1, P_2}^{\alpha,\beta}(t)
-\frac{d}{dt}\partial_3 F\left\{y\right\}_{P_1, P_2}^{\alpha,\beta}(t)
-A_{P_1^*}^\alpha\partial_4 F\left\{y\right\}_{P_1, P_2}^{\alpha,\beta}(t)\\
+K_{P_2^*}^\beta\partial_5 F\left\{y\right\}_{P_1, P_2}^{\alpha,\beta}(t)\Biggr)dt=0.
\end{multline*}
We obtain \eqref{eq:eqELCaputo} applying the fundamental
lemma of the calculus of variations.
\end{proof}

\begin{remark}
If the functional \eqref{eq:31} does not depend on $K_{P_2}^\beta y(t)$
and $B_{P_1}^\alpha y(t)$, then Theorem~\ref{theorem:ELCaputo}
reduces to the classical result: if $y$ is a solution to the problem
\begin{equation*}
\int\limits_a^b F\left(t,y(t),y'(t)\right)dt \longrightarrow \textrm{extr},
\quad y(a)=y_a, \quad y(b)=y_b,
\end{equation*}
then $y$ satisfies the Euler--Lagrange equation
\begin{equation}
\label{eq:CEL}
\partial_2 F\left(t,y(t),y'(t)\right)
- \frac{d}{dt} \partial_3 F\left(t,y(t),y'(t)\right)=0,
\end{equation}
$t \in [a,b]$.
\end{remark}

\begin{remark}
In the particular case when functional \eqref{eq:31}
does not depend on the integer derivative of function $y$,
we obtain from Theorem~\ref{theorem:ELCaputo} the following result:
if $y$ is a solution to the problem of extremizing
\begin{equation*}
\mathcal{J}[y]=\int\limits_a^b
F\left(t,y(t),B_{P_1}^\alpha y(t),K_{P_2}^\beta y(t)\right)dt
\end{equation*}
subject to $y(a)=y_a$ and $y(b)=y_b$,
where $\alpha,\beta \in (0,1)$ and $P_j=\langle a,t,b,p_j,q_j\rangle$,
$j = 1, 2$, then $\partial_2 F-A_{P_1^*}^\alpha\partial_3 F
+ K_{P_2^*}^\beta\partial_4 F=0$ with $P_j^*= \langle a,t,b,q_j,p_j\rangle$,
$j=1,2$. This extends some of the recent results of \cite{Shakoor:01}.
\end{remark}

\begin{remark}
\label{rem:AintoB}
The optimization problem \eqref{eq:31}--\eqref{eq:32} does not involve
the generalized Riemann--Liouville fractional derivative $A$-op
while the necessary optimality condition \eqref{eq:eqELCaputo} does.
However, using Theorem~\ref{thm:rel},
the Euler--Lagrange equation \eqref{eq:eqELCaputo}
can be written in terms of $B$-op as
\begin{multline*}
\partial_2 F \left\{y\right\}_{P_1, P_2}^{\alpha,\beta}(t)
-\frac{d}{dt}\partial_3 F\left\{y\right\}_{P_1, P_2}^{\alpha,\beta}(t)
-q \, \partial_4 F\left\{y\right\}_{P_1, P_2}^{\alpha,\beta}(a)k_{1-\alpha}(t,a)\\
+ p \, \partial_4 F\left\{y\right\}_{P_1, P_2}^{\alpha,\beta}(b) \,
k_{1-\alpha}(b,t) - B_{P_1^*}^\alpha\partial_4 F\left\{y\right\}_{P_1, P_2}^{\alpha,\beta}(t)
+K_{P_2^*}^\beta\partial_5 F\left\{y\right\}_{P_1, P_2}^{\alpha,\beta}(t)=0.
\end{multline*}
\end{remark}

\begin{corollary}
Let $0<\alpha,\beta<1$. If $y$ is a solution to the problem
\begin{equation*}
\int\limits_a^b F\left(t,y(t),y'(t), {_{a}^{C}\textsl{D}_t^\alpha} y(t),
{_{a}\textsl{I}_t}^\beta y(t)\right)dt \longrightarrow \min_{y\in Lip},
\quad y(a)=y_a, \quad y(b)=y_b,
\end{equation*}
then the following Euler--Lagrange equation holds:
\begin{multline*}
{_{t}}\textsl{D}_b^\alpha \partial_4 F\left(t,y(t),y'(t),
_{a}^{C}\textsl{D}_t^\alpha y(t),_{a}\textsl{I}_t^\beta y(t)\right)
+{_{t}}\textsl{I}_b^\beta \partial_5 F\left(t,y(t),y'(t),
_{a}^{C}\textsl{D}_t^\alpha y(t),_{a}\textsl{I}_t^\beta y(t)\right)\\
+ \partial_2 F\left(t,y(t),y'(t), _{a}^{C}\textsl{D}_t^\alpha y(t),
_{a}\textsl{I}_t^\beta y(t)\right) - \frac{d}{dt}\partial_3 F\left(t,y(t),y'(t),
_{a}^{C}\textsl{D}_t^\alpha y(t),_{a}\textsl{I}_t^\beta y(t)\right) = 0.
\end{multline*}
\end{corollary}

\begin{proof}
The intended Euler--Lagrange equation follows
from \eqref{eq:eqELCaputo} by choosing the $p$-sets
$P_1=P_2=\langle a,t,b,1,0\rangle$ and the kernel
$k_{1-\alpha}(t-\tau)=\frac{1}{\Gamma(1-\alpha)}(t-\tau)^{-\alpha}$.
\end{proof}

\begin{corollary}
Let $\mathcal{J}$ be the functional
$$
\mathcal{J}[y] = \int\limits_a^b F\left(t,y(t),y'(t),
p \, _{a}^{C}\textsl{D}_t^\alpha y(t)
+q \, _{t}^{C}\textsl{D}_b^\alpha y(t)\right)dt,
$$
where $p$ and $q$ are real numbers,
and $y$ be an extremizer of $\mathcal{J}$ satisfying boundary conditions
$y(a)=y_a$ and $y(b)=y_b$. Then, $y$ satisfies the Euler--Lagrange equation
\begin{equation}
\label{eq:35}
p \, {_{t}}\textsl{D}_b^\alpha \partial_4 F
+ q \, {_{a}}\textsl{D}_t^\alpha \partial_4 F
+\partial_2 F - \frac{d}{dt}\partial_3 F = 0
\end{equation}
with functions evaluated at
$\left(t,y(t),y'(t), p \, _{a}^{C}\textsl{D}_t^\alpha y(t)
+q \, _{t}^{C}\textsl{D}_b^\alpha y(t)\right)$, $t \in [a,b]$.
\end{corollary}

\begin{proof}
Choose $P_1=\langle a,t,b,p,-q\rangle$ and $k_{1-\alpha}(t-\tau)
=\frac{1}{\Gamma(1-\alpha)}(t-\tau)^{-\alpha}$.
Then the $B$-op reduces to the sum of the left and right Caputo fractional derivatives
and \eqref{eq:35} follows from \eqref{eq:eqELCaputo}.
\end{proof}


\subsubsection{Free initial boundary}
\label{sec:fib}

Let in problem \eqref{eq:31}--\eqref{eq:32} the value
of the unknown function $y$ be not preassigned
at the initial point $t=a$, \textrm{i.e.},
\begin{equation}
\label{eq:NatBound2}
y(a) \textnormal{ is free } \textnormal{ and } y(b)=y_b.
\end{equation}
Then, we do not require $\eta$ in the proof
of Theorem~\ref{theorem:ELCaputo} to vanish at $t=a$.
Therefore, following the proof
of Theorem~\ref{theorem:ELCaputo}, we obtain
\begin{multline}
\label{eq:36}
\eta(a)\partial_3 F\left\{y\right\}_{P_1, P_2}^{\alpha,\beta}(a)
+\eta(a)\left.K_{P_1^*}^{1-\alpha}\partial_4
F\left\{y\right\}_{P_1, P_2}^{\alpha,\beta}(t)\right|_{t=a}\\
+\int_a^b\eta(t)\Biggl(\partial_2 F \left\{y\right\}_{P_1, P_2}^{\alpha,\beta}(t)
-\frac{d}{dt}\partial_3 F\left\{y\right\}_{P_1, P_2}^{\alpha,\beta}(t)
-A_{P_1^*}^\alpha\partial_4 F\left\{y\right\}_{P_1, P_2}^{\alpha,\beta}(t)\\
+K_{P_2^*}^\beta\partial_5 F\left\{y\right\}_{P_1, P_2}^{\alpha,\beta}(t)\Biggr)dt=0
\end{multline}
for every admissible $\eta\in Lip([a,b];\mathbb{R})$ with $\eta(b)=0$.
In particular, condition \eqref{eq:36} holds for those $\eta$ that fulfill $\eta(a)=0$.
Hence, by the fundamental lemma of the calculus of variations, equation \eqref{eq:eqELCaputo}
is satisfied. Now, let us return to \eqref{eq:36} and let $\eta$
again be arbitrary at point $t=a$. Using equation \eqref{eq:eqELCaputo},
we obtain the following natural boundary condition:
\begin{equation}
\label{eq:NatBoundCond2}
\partial_3 F\left\{y\right\}_{P_1, P_2}^{\alpha,\beta}(a)
+\left.K_{P_1^*}^{1-\alpha}\partial_4
F\left\{y\right\}_{P_1, P_2}^{\alpha,\beta}(t)\right|_{t=a}=0.
\end{equation}

We just obtained the following result.

\begin{theorem}
\label{thm:tc:a}
If $y \in Lip([a,b];\mathbb{R})$ is an extremizer of \eqref{eq:31}
subject to $y(b) = y_b$, then $y$ satisfies the Euler--Lagrange
equation \eqref{eq:eqELCaputo} and the transversality condition
\eqref{eq:NatBoundCond2}.
\end{theorem}

\begin{corollary}[\textrm{cf.} Theorem 2.3 of \cite{Isoperimetric}]
Let $\mathcal{J}$ be the functional given by
\begin{equation*}
\mathcal{J}[y]=\int\limits_a^b
F\left(t,y(t),_{a}^{C}\textsl{D}_t^\alpha y(t)\right)dt.
\end{equation*}
If $y$ is a local minimizer of $\mathcal{J}$ satisfying the boundary
condition $y(b)=y_b$, then $y$ satisfies the Euler--Lagrange equation
\begin{equation}
\label{eq:37}
\partial_2 F\left(t,y(t),_{a}^{C}\textsl{D}_t^\alpha y(t)\right)
+_{t}\textsl{D}_b^\alpha\partial_3 F\left(t,y(t),_{a}^{C}\textsl{D}_t^\alpha y(t)\right)=0
\end{equation}
and the natural boundary condition
\begin{equation}
\label{eq:38}
\left._{t}\textsl{I}_b^{1-\alpha}\partial_3
F\left(t,y(t),_{a}^{C}\textsl{D}_t^\alpha y(t)\right)\right|_{t=a}=0.
\end{equation}
\end{corollary}

\begin{proof}
Let functional \eqref{eq:31} be such that it does not depend on the integer
derivative $y'(t)$ and on $K$-op. If $P_1=\langle a,t,b,1,0\rangle$ and
$k_{1-\alpha}(t-\tau)=\frac{1}{\Gamma(1-\alpha)}(t-\tau)^{-\alpha}$,
then $B$-op reduces to the left Caputo fractional derivative and from
\eqref{eq:eqELCaputo} and \eqref{eq:NatBoundCond2}
we deduce \eqref{eq:37} and \eqref{eq:38}, respectively.
\end{proof}

\begin{remark}
Observe that if the functional \eqref{eq:31} is independent of $K$-op,
then the problem defined by \eqref{eq:31} and \eqref{eq:NatBound2}
takes the form
\begin{equation*}
\int\limits_a^b F\left(t,y(t),y'(t),B_{P_1}^\alpha y(t)\right)dt
\longrightarrow \textrm{extr}, \quad y(b)=y_b
\end{equation*}
($y(a)$ free) and the optimality conditions \eqref{eq:eqELCaputo}
and \eqref{eq:NatBoundCond2} reduce respectively to
\begin{multline*}
\partial_2 F \left(t,y(t),y'(t),B_{P_1}^\alpha y(t)\right)
-\frac{d}{dt}\partial_3 F\left(t,y(t),y'(t),B_{P_1}^\alpha y(t)\right)\\
- A_{P_1^*}^\alpha\partial_4 F\left(t,y(t),y'(t),B_{P_1}^\alpha y(t)\right)=0
\end{multline*}
and $\partial_3 F\left(a,y(a),y'(a),B_{P_1}^\alpha y(a)\right)
+\left.K_{P_1^*}^{1-\alpha}\partial_4
F\left(t,y(t),y'(t),B_{P_1}^\alpha y(t)\right)\right|_{t=a}=0$.
\end{remark}


\subsubsection{Isoperimetric problems}
\label{sub:gen:iso:fp}

One of the earliest problems in geometry was the isoperimetric problem,
already considered by the ancient Greeks. It consists to find,
among all closed curves of a given length, the one which encloses the
maximum area. The general problem for which one integral is to be given
a fixed value, while another is to be made a maximum or a minimum,
is nowadays part of the calculus of variations \cite{BorisBookI,BorisBookII}.
Such \emph{isoperimetric problems} have found a broad class of important applications
throughout the centuries, with numerous useful implications
in astronomy, geometry, algebra, analysis,
and engineering \cite{Viktor,Curtis}.
For recent advancements on the study of isoperimetric problems
see \cite{isoJMAA,isoNabla,iso:ts} and references therein.
Here we consider isoperimetric problems
with generalized fractional operators.
Similarly to Sections~\ref{sec:fp} and \ref{sec:fib},
we deal with integrands involving both generalized
Caputo fractional derivatives and generalized fractional integrals,
as well as the classical derivative.

Let $0<\alpha,\beta<1$ and $P_j=\langle a,t,b,p_j,q_j\rangle$,
$j=1,2$, be given $p$-sets. Consider the following isoperimetric problem:
\begin{gather}
\mathcal{J}[y]=\int\limits_a^b F
\left\{y\right\}_{P_1, P_2}^{\alpha,\beta}(t)dt \longrightarrow \textrm{extr},\label{eq:IsoFunct1} \\
y(a)=y_a, \quad y(b)=y_b, \label{eq:IsoBoun1}\\
\mathcal{I}[y]=\int\limits_a^b G\left\{y\right\}_{P_1, P_2}^{\alpha,\beta}(t)dt= \xi. \label{eq:IsoConstr1}
\end{gather}

\begin{definition}
A Lipschitz function $y : [a,b]\to\mathbb R$
is said to be \emph{admissible} for problem
\eqref{eq:IsoFunct1}--\eqref{eq:IsoConstr1}
if it satisfies the given boundary conditions \eqref{eq:IsoBoun1}
and the isoperimetric constraint \eqref{eq:IsoConstr1}.
\end{definition}

We assume that $F,G\in C^1([a,b] \times \mathbb{R}^4;\mathbb{R})$,
$\xi$ is a specified real constant, functions
$t \mapsto \partial_4 F\left\{y\right\}_{P_1,P_2}^{\alpha,\beta}(t)$
and $t \mapsto \partial_4 G\left\{y\right\}_{P_1,P_2}^{\alpha,\beta}(t)$
are absolutely continuous and have continuous derivatives $A_{P_1^*}^\alpha$,
and functions $t \mapsto \partial_3 F\left\{y\right\}_{P_1,P_2}^{\alpha,\beta}(t)$
and $t \mapsto \partial_3 G\left\{y\right\}_{P_1,P_2}^{\alpha,\beta}(t)$
have continuous derivatives $\frac{d}{dt}$.

\begin{definition}
An admissible function $y\in Lip\left([a,b],\mathbb{R}\right)$
is an \emph{extremal} for $\mathcal{I}$ if it satisfies
the Euler--Lagrange equation \eqref{eq:eqELCaputo}
associated with \eqref{eq:IsoConstr1}, \textrm{i.e.},
\begin{multline*}
\partial_2 G \left\{y\right\}_{P_1, P_2}^{\alpha,\beta}(t)
-\frac{d}{dt}\partial_3 G\left\{y\right\}_{P_1, P_2}^{\alpha,\beta}(t)
-A_{P_1^*}^\alpha\partial_4 G\left\{y\right\}_{P_1, P_2}^{\alpha,\beta}(t)\\
+K_{P_2^*}^\beta\partial_5 G\left\{y\right\}_{P_1, P_2}^{\alpha,\beta}(t)=0,
\end{multline*}
where $P_j^*=\langle a,t,b,q_j,p_j\rangle$, $j=1,2$, and $t\in[a,b]$.
\end{definition}

The next theorem gives a necessary optimality condition for the
generalized fractional isoperimetric problem
\eqref{eq:IsoFunct1}--\eqref{eq:IsoConstr1}.

\begin{theorem}
\label{theorem:EL2}
If $y$ is a solution to the isoperimetric problem
\eqref{eq:IsoFunct1}--\eqref{eq:IsoConstr1} and
is not an extremal for $\mathcal{I}$, then
there exists a real constant $\lambda$ such that
\begin{multline}
\label{eq:eqEL2}
\partial_2 H \left\{y\right\}_{P_1, P_2}^{\alpha,\beta}(t)
-\frac{d}{dt}\partial_3 H\left\{y\right\}_{P_1, P_2}^{\alpha,\beta}(t)
-A_{P_1^*}^\alpha\partial_4 H\left\{y\right\}_{P_1, P_2}^{\alpha,\beta}(t)\\
+K_{P_2^*}^\beta\partial_5 H\left\{y\right\}_{P_1, P_2}^{\alpha,\beta}(t)=0,
\end{multline}
$t\in[a,b]$, where $H(t,y,u,v,w)=F(t,y,u,v,w)-\lambda G(t,y,u,v,w)$.
\end{theorem}

\begin{proof}
Consider a two-parameter family of the form
$\hat{y}=y+\varepsilon_1\eta_1+\varepsilon_2\eta_2$,
where for each $i\in\{1,2\}$ we have $\eta_i(a)=\eta_i(b)=0$.
First we show that we can select $\eta_2$ such that
$\hat{y}$ satisfies \eqref{eq:IsoConstr1}. Consider the quantity
\begin{multline*}
\mathcal{I}[\hat{y}]=\int\limits_a^b G\biggr(t,y(t)+\varepsilon_1\eta_1(t)
+\varepsilon_2\eta_2(t),\frac{d}{dt}\left(y(t)+\varepsilon_1\eta_1(t)
+\varepsilon_2\eta_2(t)\right),\\
B_{P_1}^\alpha \left(y(t)+\varepsilon_1\eta_1(t)
+\varepsilon_2\eta_2(t)\right),K_{P_2}^\beta \left(y(t)+\varepsilon_1\eta_1(t)
+\varepsilon_2\eta_2(t)\right)\biggr)dt.
\end{multline*}
Looking to $\mathcal{I}[\hat{y}]$ as a function
of $\varepsilon_1,\varepsilon_2$, we define
$\hat{I}(\varepsilon_1,\varepsilon_2):=\mathcal{I}[\hat{y}]-\xi$.
Thus, $\hat{I}(0,0)=0$.
On the other hand, integrating by parts, we obtain
\begin{multline*}
\left.\frac{\partial\hat{I}}{\partial\varepsilon_2}\right|_{(0,0)}
=\int\limits_a^b\eta_2(t)\biggl(\partial_2 G \left\{y\right\}_{P_1, P_2}^{\alpha,\beta}(t)
-\frac{d}{dt}\partial_3 G\left\{y\right\}_{P_1, P_2}^{\alpha,\beta}(t)\\
-A_{P_1^*}^\alpha\partial_4 G\left\{y\right\}_{P_1, P_2}^{\alpha,\beta}(t)
+K_{P_2^*}^\beta\partial_5 G\left\{y\right\}_{P_1, P_2}^{\alpha,\beta}(t)\biggr)dt,
\end{multline*}
where $P_j^*=\langle a,t,b,q_j,p_j\rangle$, $j=1,2$. We assumed
that $y$ is not an extremal for $\mathcal{I}$.
Hence, the fundamental lemma of the calculus of variations implies
that there exists a function $\eta_2$ such that
$\left.\frac{\partial\hat{I}}{\partial\varepsilon_2}\right|_{(0,0)}\neq 0$.
According to the implicit function theorem, there exists
a function $\varepsilon_2(\cdot)$ defined in a neighborhood of $0$ such that
$\hat{I}(\varepsilon_1,\varepsilon_2(\varepsilon_1))=0$.
Let $\hat{J}(\varepsilon_1,\varepsilon_2)=\mathcal{J}[\hat{y}]$.
Function $\hat{J}$ has an extremum at $(0,0)$ subject to $\hat{I}(0,0)=0$,
and we have proved that $\nabla\hat{I}(0,0)\neq 0$. The Lagrange multiplier rule
asserts that there exists a real number $\lambda$ such that
$\nabla(\hat{J}(0,0)-\lambda\hat{I}(0,0))=0$. Because
\begin{multline*}
\left.\frac{\partial\hat{J}}{\partial\varepsilon_1}\right|_{(0,0)}
=\int\limits_a^b\biggl(\partial_2 F \left\{y\right\}_{P_1, P_2}^{\alpha,\beta}(t)
-\frac{d}{dt}\partial_3 F\left\{y\right\}_{P_1, P_2}^{\alpha,\beta}(t)\\
-A_{P_1^*}^\alpha\partial_4 F\left\{y\right\}_{P_1, P_2}^{\alpha,\beta}(t)
+K_{P_2^*}^\beta\partial_5 F\left\{y\right\}_{P_1, P_2}^{\alpha,\beta}(t)\biggr)\eta_1(t) dt
\end{multline*}
and
\begin{multline*}
\left.\frac{\partial\hat{I}}{\partial\varepsilon_1} \right|_{(0,0)}
=\int\limits_a^b\biggl(\partial_2 G \left\{y\right\}_{P_1, P_2}^{\alpha,\beta}(t)
-\frac{d}{dt}\partial_3 G\left\{y\right\}_{P_1, P_2}^{\alpha,\beta}(t)\\
-A_{P_1^*}^\alpha\partial_4 G\left\{y\right\}_{P_1, P_2}^{\alpha,\beta}(t)
+K_{P_2^*}^\beta\partial_5 G\left\{y\right\}_{P_1, P_2}^{\alpha,\beta}(t)\biggr)\eta_1 (t)dt,
\end{multline*}
one has
\begin{multline*}
\int\limits_a^b\biggl(\partial_2 H \left\{y\right\}_{P_1, P_2}^{\alpha,\beta}(t)
-\frac{d}{dt}\partial_3 H\left\{y\right\}_{P_1, P_2}^{\alpha,\beta}(t)\\
-A_{P_1^*}^\alpha\partial_4 H\left\{y\right\}_{P_1, P_2}^{\alpha,\beta}(t)
+K_{P_2^*}^\beta\partial_5 H\left\{y\right\}_{P_1, P_2}^{\alpha,\beta}(t)\biggr)\eta_1(t) dt=0.
\end{multline*}
We get equation \eqref{eq:eqEL2} from the fundamental lemma of the calculus of variations.
\end{proof}

As particular cases of our problem \eqref{eq:IsoFunct1}--\eqref{eq:IsoConstr1},
one obtains previously studied fractional isoperimetric problems
with Caputo derivatives.

\begin{corollary}[\textrm{cf.} Theorem~3.3 of \cite{Isoperimetric}]
\label{IsoPro:RicDel}
Let $y$ be a local minimizer to
\begin{gather*}
\mathcal{J}[y]=\int_a^b L\left(t,y(t),\, {_a^C D_t^\alpha} y(t)\right) dt
\longrightarrow \min, \\
\mathcal{I}[y]=\int_a^b g\left(t,y(t),\, {_a^C D_t^\alpha} y(t)\right)dt=\xi,\\
y(a)=y_a,\ y(b)=y_b.
\end{gather*}
If $y$ is not an extremal of $\mathcal{I}$,
then there exists a constant $\lambda$ such that $y$ satisfies
$\partial_2 F\left(t,y(t),\, {_a^C D_t^\alpha} y(t)\right)
+{_tD_b^\alpha}
\partial_3 F\left(t,y(t),\, {_a^C D_t^\alpha} y(t)\right)=0$,
$t \in [a,b]$, with $F=L+\lambda g$.
\end{corollary}

\begin{proof}
The result follows from Theorem~\ref{theorem:EL2} by choosing the kernel
$k_{1-\alpha}(t-\tau)=\frac{1}{\Gamma(1-\alpha)}(t-\tau)^{-\alpha}$
and the $p$-set $P_1$ to be $P_1=\langle a,t,b,1,0\rangle$.
Indeed, in this case the operator $-A_{P^*}^\alpha$ becomes the right
Riemann--Liouville fractional derivative, and the operator $B_P^\alpha$ becomes
the left Caputo fractional derivative.
\end{proof}

\begin{remark}
If functionals \eqref{eq:IsoFunct1} and \eqref{eq:IsoConstr1}
do not depend on integer derivatives, then problem
\eqref{eq:IsoFunct1}--\eqref{eq:IsoConstr1} is reduced
to extremize functional
\begin{equation*}
\mathcal{J}[y]=\int\limits_a^b
F\left(t,y(t),B_{P_1}^\alpha y(t),K_{P_2}^\beta y(t)\right)dt
\end{equation*}
subject to boundary conditions
$y(a)=y_a$, $y(b)=y_b$,
and the isoperimetric constraint
\begin{equation*}
\mathcal{I}[y]=\int\limits_a^b
G\left(t,y(t),B_{P_1}^\alpha y(t),K_{P_2}^\beta y(t)\right)dt.
\end{equation*}
By \eqref{eq:eqEL2} there exists $\lambda$ such that $y$ satisfies
$\partial_2 H -A_{P_1^*}^\alpha\partial_4 H
+K_{P_2^*}^\beta\partial_5 H=0$,
$t\in[a,b]$, with $H=F-\lambda G$.
\end{remark}

\begin{remark}
Theorem~\ref{theorem:EL2} can be extended to the case when
$y$ is an extremal for $\mathcal{I}$. The proof is similar but one needs
to use the extended (abnormal) Lagrange multiplier rule.
The method is given in \cite{isoNabla}.
\end{remark}


\subsection{The coherence embedding problem}
\label{sec:coh}

The notion of embedding introduced in \cite{cd} is an algebraic procedure
providing an extension of classical differential equations over
an arbitrary vector space. This formalism is developed in the framework
of stochastic processes \cite{cd}, non-differentiable functions \cite{cft},
discrete sets \cite{BCGI}, and fractional equations \cite{Cresson}.
The general scheme of embedding theories is the following:
(i) fix a vector space $V$ and a mapping
$\iota : C^0 ([a,b] ,\mathbb{R}^n ) \rightarrow V$;
(ii) extend differential operators over $V$;
(iii) extend the notion of integral over $V$.
Let $(\iota , D,J)$ be a given embedding formalism, where a linear operator
$D : V \rightarrow V$ takes place for a generalized derivative on $V$, and
a linear operator $J: V \rightarrow \mathbb{R}$ takes place for a generalized
integral on $V$. The embedding procedure gives two different ways, a priori, to generalize
Euler--Lagrange equations. The first (pure algebraic) way is to make a direct embedding
of the Euler--Lagrange equation. The second (analytic) is to embed the Lagrangian
functional associated to the equation and to derive, by the associated calculus of variations,
the Euler--Lagrange equation for the embedded functional. A natural
question is then the problem of coherence between these two extensions:

{\sc Coherence problem}.
{\it Let $(\iota , D,J)$ be a given embedding formalism.
Do we have equivalence between the Euler--Lagrange equation which gives the direct embedding
and the one received from the embedded Lagrangian system?}

For the standard fractional differential calculus of Riemann--Liouville or Caputo,
the answer to the question above is known to be negative.
For a gentle explanation of the fractional embedding and its importance,
we refer the reader to \cite{Cresson,PhDInizan,klimek}.
Here we propose a coherent embedding in the framework of our fractional generalized
calculus by choosing the generalized fractional operator to be $K_P^\alpha$
with $q=-p$. A direct embedding of the classical
Euler--Lagrange equation \eqref{eq:CEL} gives
\begin{equation}
\label{eq:d:e}
\partial_2 F \left(t,y(t),K_P^\alpha y(t)\right)
-K_{P}^{\alpha} \partial_3 F\left(t,y(t),K_P^{\alpha} y(t)\right)=0
\end{equation}
for $t\in[a,b]$. On the other hand, we can apply Theorem~\ref{theorem:ELCaputo}
to the embedded Lagrangian functional
$\mathcal{J}[y]=\int\limits_a^b F\left(t,y(t),K_P^{\alpha} y(t)\right)dt$.
Let $P=\langle a,t,b,p,-p\rangle$ and $\alpha\in(0,1)$.
If $y$ is a solution to the problem
\begin{equation}
\label{eq:prb:coh}
\begin{gathered}
\mathcal{J}[y]=\int\limits_a^b F\left(t,y(t),K_P^{\alpha} y(t)\right)dt
\longrightarrow \textrm{extr},\\
y(a)=y_a, \quad y(b)=y_b,
\end{gathered}
\end{equation}
then, by Theorem~\ref{theorem:ELCaputo},
$y$ satisfies the Euler--Lagrange equation given by
\begin{equation}
\label{eq:1}
\partial_2 F \left(t,y(t),K_P^{\alpha} y(t)\right)
+K_{P^*}^{\alpha}\partial_3 F\left(t,y(t),K_P^{\alpha} y(t)\right)=0,
\end{equation}
$t\in[a,b]$. For an arbitrary kernel $k_\alpha$, an easy computation shows that
for $p = -q$ one has $K_P^{\alpha} f(t)=-K_{P^*}^{\alpha} f(t)$.
Therefore, equation \eqref{eq:1} can be written in the form
\begin{equation}
\label{eq:lap}
\partial_2 F \left(t,y(t),K_P^{\alpha} y(t)\right)
-K_{P}^{\alpha}\partial_3 F\left(t,y(t),K_P^{\alpha} y(t)\right)=0,
\end{equation}
$t\in[a,b]$. It means that the Euler--Lagrange equation \eqref{eq:d:e}
obtained by the direct fractional embedding procedure
and the Euler--Lagrange equation \eqref{eq:lap} obtained by the
least action principle coincide. We just proved the following result.

\begin{theorem}
\label{thm:coherent}
Let $k_\alpha(t,\tau)$ be an arbitrary kernel
and $P$ a $p$-set with $q=-p$:
$P=\langle a,t,b,p,-p\rangle$. Then the fractional variational
problem \eqref{eq:prb:coh} is coherent.
\end{theorem}


\section{Illustrative examples}
\label{sec:ex}

In this section we illustrate our results through two examples
of isoperimetric problems with different kernels.
Explicit expressions for the minimizers are given.

In Example~\ref{ex:1} we make use of the Mittag--Leffler function
of two parameters. Let $\alpha, \beta>0$. We recall that
the Mittag--Leffler function is defined by
\begin{equation*}
E_{\alpha,\beta}(z)
=\sum_{k=0}^\infty\frac{z^k}{\Gamma(\alpha k+\beta)}\, .
\end{equation*}
This function appears naturally in the solution
of fractional differential equations,
as a generalization of the exponential function
\cite{CapelasOliveira}.
Indeed, while a linear second order
ordinary differential equation
with constant coefficients presents an exponential function as solution,
in the fractional case the Mittag--Leffler functions emerge \cite{book:Kilbas}.

\begin{example}
\label{ex:1}
Let $\alpha\in\left(0,1\right)$ and $\xi \in\mathbb{R}$.
Consider the following problem:
\begin{equation}
\label{eq:ex}
\begin{gathered}
\mathcal{J}(y)=\int_0^1\left(y'
+ \, {\textsl{B}_P^\alpha} y\right)^2dt \longrightarrow \min,\\
\mathcal{I}(y)=\int_0^1\left(y'
+ \, {\textsl{B}_P^\alpha} y\right)dt = \xi,\\
y(0)=0, \quad
y(1)=\int_0^1 E_{1-\alpha,1}\left(
-\left(1-\tau\right)^{1-\alpha}\right) \xi d\tau,
\end{gathered}
\end{equation}
where $k_{1-\alpha}(t-\tau)=\frac{1}{\Gamma(1-\alpha)}(t-\tau)^{-\alpha}$
and $P=\langle 0,t,1,1,0\rangle$. In this case the $B$-op
becomes the left Caputo fractional derivative,
and the augmented Lagrangian $H$ of Theorem~\ref{theorem:EL2} is given by
$H(t,y,v,w) =(v+w)^2 -\lambda (v+w)$. One can easily check that
\begin{equation}
\label{eq:y:ex}
y(t)=\int_0^t E_{1-\alpha,1}\left(-\left(t
-\tau\right)^{1-\alpha}\right)\xi d\tau
\end{equation}
is not an extremal for $\mathcal{I}$
and satisfies $y'+\,{\textsl{B}_P^\alpha} y= \xi$.
Moreover, \eqref{eq:y:ex} satisfies
\eqref{eq:eqEL2} for $\lambda=2\xi$, \textrm{i.e.},
\begin{equation*}
-\frac{d}{dt}\left(2\left(y'+ \,{\textsl{B}_P^\alpha} y\right) -2\xi\right)
- \, {\textsl{A}_{P^*}^\alpha}\left(2\left(y'
+ \, {\textsl{B}_P^\alpha} y\right) -2\xi\right)=0,
\end{equation*}
where $P^*=\langle 0,t,1,0,1\rangle$ is the dual $p$-set
of $P$. We conclude that \eqref{eq:y:ex}
is an extremal for problem \eqref{eq:ex}.
Since in this example one has a
problem \eqref{eq:IsoFunct1}--\eqref{eq:IsoConstr1}
with $F(t,y,v,w) = (v+w)^2$ and $G(t,y,v,w) = v+w$, simple
convexity arguments show (see \cite[Section~6]{Integrals}
and \cite[Section~3.4]{Isoperimetric}) that \eqref{eq:y:ex}
is indeed the global minimizer to problem \eqref{eq:ex}.
\end{example}

\begin{example}
Let $\alpha\in\left(0,1\right)$,
$P=\langle 0,t,1,1,0\rangle$.
Consider the following problem:
\begin{equation*}
\begin{gathered}
\mathcal{J}(y)=\int_0^1\left({\textsl{K}_P^\alpha} y
+t\right)^2dt \longrightarrow \min,\\
\mathcal{I}(y)=\int_0^1 t{\textsl{K}_P^\alpha} ydt = \xi,\\
y(0)=\xi-1, \quad y(1)=(\xi-1)\left(1
+\int_0^1 r_{\alpha}(1-\tau) d\tau\right),
\end{gathered}
\end{equation*}
where the kernel is such that $k_\alpha(t,\tau)=k_\alpha(t-\tau)$
with $k_\alpha(0)=1$ and $\textsl{K}_{P^*}^\alpha t\neq 0$.
The resolvent $r_{\alpha}(t)$ is given by $r_{\alpha}(t)
=\mathcal{L}^{-1}\left[\frac{1}{s\widetilde{k}_\alpha(s)}-1\right]$,
$\widetilde{k}_\alpha(s)=\mathcal{L}\left[k_\alpha(t)\right]$, where
$\mathcal{L}$ and $\mathcal{L}^{-1}$ are the direct and inverse Laplace transforms,
respectively. Since $\textsl{K}_{P^*}^\alpha t\neq 0$, there is no solution
to the Euler--Lagrange equation for functional $\mathcal{I}$.
The augmented Lagrangian $H$ of Theorem~\ref{theorem:EL2} is given by
$H(t,y,w) =(w+t)^2 -\lambda tw$. Function
\begin{equation*}
y(t) = \left(\xi-1\right) \left( 1 +\int_0^t r_{\alpha}(t-\tau)d\tau \right)
\end{equation*}
is the solution to the Volterra integral equation of the first kind
$\textsl{K}_{P}^\alpha y=(\xi-1)t$ (see, \textrm{e.g.}, Eq. 16, p.~114
of \cite{book:Polyanin}) and for $\lambda=2\xi$ satisfies
our optimality condition \eqref{eq:eqEL2}:
\begin{equation}
\label{eq:noc:ex2}
\textsl{K}_{P^*}^\alpha\left(2\left(\textsl{K}_P^\alpha y+t\right)
-2\xi t\right)=0.
\end{equation}
The solution of \eqref{eq:noc:ex2} subject to the given boundary conditions
depends on the particular choice for the kernel. For example, let
$k_{\alpha}(t-\tau)=e^{\alpha(t-\tau)}$. Then the solution of \eqref{eq:noc:ex2}
subject to the boundary conditions $y(0)=\xi-1$ and $y(1)=(\xi-1)(1-\alpha)$
is $y(t)=(\xi-1)(1-\alpha t)$ (\textrm{cf.} \cite[p.~15]{book:Polyanin}).
If $k_{\alpha}(t-\tau)=\cos\left(\alpha(t-\tau)\right)$, then the boundary
conditions are $y(0)=\xi-1$ and $y(1)=(\xi-1)\left(1+\alpha^2/2\right)$,
and the extremal is $y(t)=(\xi-1)\left(1+\alpha^2 t^2/2\right)$
(\textrm{cf.} \cite[p.~46]{book:Polyanin}).
\end{example}

Borrowing different kernels from book \cite{book:Polyanin},
many other examples of dynamic optimization problems can be
explicitly solved by application of the results of Section~\ref{sec:mr}.


\section*{Acknowledgements}

Work supported by {\it FEDER} funds through
{\it COMPETE} --- Operational Programme Factors of Competitiveness
(``Programa Operacional Factores de Competitividade'')
and by Portuguese funds through the
{\it Center for Research and Development
in Mathematics and Applications} (University of Aveiro)
and the Portuguese Foundation for Science and Technology
(``FCT --- Funda\c{c}\~{a}o para a Ci\^{e}ncia e a Tecnologia''),
within project PEst-C/MAT/UI4106/2011
with COMPETE number FCOMP-01-0124-FEDER-022690.
Odzijewicz was also supported by FCT through the Ph.D. fellowship
SFRH/BD/33865/2009; Malinowska by Bia{\l}ystok
University of Technology grant S/WI/02/2011
and by the European Union Human Capital Programme
\emph{Podniesienie potencja{\l}u uczelni wyzszych jako czynnik rozwoju gospodarki opartej na wiedzy};
and Torres by FCT through the project PTDC/MAT/113470/2009.



\end{document}